\documentclass[letterpaper, 10 pt, conference, twocolumn]{root}  

\IEEEoverridecommandlockouts                              

\usepackage{graphicx}
\usepackage{mathrsfs,amsmath}
\usepackage[hidelinks]{hyperref}
\usepackage{amssymb}
\usepackage{amsfonts}
\usepackage{amsthm}
\usepackage{mathtools}
\usepackage{caption}
\usepackage{subcaption}
\usepackage{algorithm}
\usepackage{algorithmic}
\usepackage{float}

\newtheorem{theorem}{Theorem}

\newtheorem{assumption}{Assumption}
\theoremstyle{definition}

\usepackage{cite}

 \usepackage{pgfplots}
  \pgfplotsset{compat=newest}
  \usetikzlibrary{plotmarks}
  \usetikzlibrary{arrows.meta}
  \usepgfplotslibrary{patchplots}
  \usepackage{grffile}
  \usepackage{amsmath}
  \usepackage{balance}
\colorlet{savedleftcolor}{.}

\title{Robust Control Barrier Function Design for High Relative Degree Systems: Application to Unknown Moving Obstacle Collision Avoidance}

\author{Kwang Hak Kim, Mamadou Diagne, and Miroslav Krstić
\thanks{This work was supported by the Department of the Navy, Office of Naval Research under grant N00014-23-1-2376. The results and opinions in this paper are solely of the authors and does not reflect the position or the policy of the U.S. Government.\endgraf
K. Kim, M. Diagne, and M. Krstić is with the Department of Mechanical and Aerospace Engineering, U.C. San Diego, 9500 Gilman Drive, La Jolla, CA, 92093-0411, {\tt\small \{kwk001,mdiagne,krstic\}@ucsd.edu}.}
}
\begin{document}

\maketitle
\thispagestyle{empty}
\pagestyle{empty}


\begin{abstract}

   In safety-critical control, managing safety constraints with high relative degrees and uncertain obstacle dynamics pose significant challenges in guaranteeing safety performance. Robust Control Barrier Functions (RCBFs) offer a potential solution, but the non-smoothness of the standard RCBF definition can pose a challenge when dealing with multiple derivatives in high relative degree problems. As a result, the definition was extended to the marginally more conservative smooth Robust Control Barrier Functions (sRCBF). Then, by extending the sRCBF framework to the CBF backstepping method, this paper offers a novel approach to these problems. Treating obstacle dynamics as disturbances, our approach reduces the requirement for precise state estimations of the obstacle to an upper bound on the disturbance, which simplifies implementation and enhances the robustness and applicability of CBFs in dynamic and uncertain environments. Then, we validate our technique through an example problem in which an agent, modeled using a kinematic unicycle model, aims to avoid an unknown moving obstacle. The demonstration shows that the standard CBF backstepping method is not sufficient in the presence of a moving obstacle, especially with unknown dynamics. In contrast, the proposed method successfully prevents the agent from colliding with the obstacle, proving its effectiveness.
\end{abstract}


\section{Introduction}


A prominent method in safety-critical control for dynamical systems that has gained traction is the concept of Control Barrier Functions (CBFs) introduced in \cite{ames2019cbf_theory,wieland_constructive_2007}. CBFs are a powerful Lyapunov-like tool in which the objective is to ensure safety. Further, \cite{ames2014cbfqp_cruise,ames2017cbfqp_critical}, has proposed the use of a quadratic program (QP) to develop safety filters from CBFs (CBF-QP) which ensures safety with minimal deviation from the nominal objective. These works have paved the way for the incorporation of safety-critical control into many applications, including multi-agent systems \cite{jankovic_multiagent_2024}, autonomous driving \cite{ames2014cbfqp_cruise}, and robotics \cite{hsu2015backstepping_cbf}.

However, implementing CBF-QP safety filters can be challenging when dealing with high and mixed relative degrees. When some or all of the control inputs are absent in the safety constraint's first derivative, it can lead to very limited agency in guaranteeing safety. To address the high relative degree issue, various techniques exist, such as exponential CBFs \cite{ames2019cbf_theory} and high order CBFs \cite{xiao2019hocbf}.  Another method is the CBF backstepping technique introduced in \cite{krstic_nonovershooting_2006}—which has been utilized in Prescribed Time Safety Filters (PTSf) \cite{abel2023prescribed_time} and Stefan Model PDEs \cite{koga_safe_2023,koga2023event}—that employs smart choices of initial gains with backstepping to maintain system safety by enforcing a single target CBF from a chain of CBFs. In the case of mixed relative degrees, methods such as artificially increasing the relative degree of lower inputs \cite{kim_safe_mixedrel_2023} or holding them constant \cite{huang_obstacle_2023} have been used.

One significant application area facing these relative degree challenges is in safety problems utilizing the kinematic unicycle model, particularly when the safety constraint is defined by positional constraints \cite{huang_obstacle_2023,thontepu_control_2023}. This often leads to the CBF-QP safety filter disregarding the steering input, thereby relying only on the forward drive input and resulting in inefficient safety solutions, especially for nonholonomic models like the unicycle. Solely relying on forward drive inputs can limit the system's ability to navigate complex environments or avoiding obstacles effectively.

Another prevalent challenge is in obstacle avoidance with uncertain dynamics. Existing approaches to moving obstacle avoidance using CBFs \cite{jiang_robust_2024, desai_clf-cbf_2022,restrepo_3d_2019} typically assume known obstacle dynamics. While methods that do consider uncertain dynamics offer strategies for handling unknown dynamics, they require an estimation of the obstacle dynamics with a bounded error  \cite{molnar_safety-critical_2023}, which may be difficult to achieve in some applications. In \cite{breeden_robust_2023}, which combined Robust CBFs and high order CBF techniques, only considered up to relative degree two and assumed the additive disturbance to be zero for CBFs with arbitrarily high relative degrees. 

In this paper, we introduce an innovative approach to tackling the high relative degree problem in the presence of disturbances by integrating the Robust CBF framework \cite{jankovic_robust_2018} with the CBF backstepping method \cite{krstic_nonovershooting_2006}. Our extended method addresses worst-case disturbances for arbitrarily high relative degrees and simplifies the challenge of unknown obstacle dynamics by treating these dynamics as disturbances. However, it is well-known that worst-case robustness considerations often lead to non-smooth functions, which can pose challenges when taking multiple derivatives. We mitigate this by introducing a marginally more conservative smooth Robust Control Barrier Function (sRCBF) definition to smoothen the safety constraints of RCBFs.

The paper is structured as follows: Section \ref{prelim} and \ref{motivation} provides standard preliminary definitions and outlines the motivational problem. Section \ref{qp-backstep} reviews the CBF backstepping design algorithm and introduces the QP problem for deriving a minimally intrusive safety filter. Section \ref{extention} extends the backstepping algorithm to sRCBFs for systems with disturbances. Finally, Section \ref{application} applies this extended methodology to a moving obstacle collision avoidance problem with simulations.

\section{Preliminaries} \label{prelim}

We start with a preliminary overview of Control Barrier Functions (CBFs) \cite{ames2019cbf_theory} and Robust Control Barrier Functions (RCBFs) \cite{jankovic_robust_2018}, outlining their definitions and key properties.

Consider the following nonlinear control affine system:
\begin{equation} \label{affine}
    \dot{x} = f(x) + g(x)u,
\end{equation}
with $x \in \mathbb{R}^{n}$ and $u \in \mathbb{R}^{m_1}$.

\textbf{Definition 1}: A scalar-valued differentiable function $h \ : \ \mathbb{R}^n \rightarrow \mathbb{R}$ with the property that $\inf_{x \in \mathbb{R}^n} h(x) < 0$ and $\sup_{x \in \mathbb{R}^n} h(x) > 0$, is defined as a \textit{Barrier Function (BF) candidate} and the set $\mathcal{C} = \{ x \in \mathbb{R}^n \; | \; h(x) \geq 0\}$ is defined as the \textit{safe set}.


\textbf{Definition 2}: A continuously differentiable function $h \; : \; \mathbb{R}^n \rightarrow \mathbb{R}$, is a \textit{Control Barrier Function (CBF)} if there exists an extended class $\mathcal{K}_{\infty}$ function $\alpha \; : \; \mathbb{R} \rightarrow \mathbb{R}$ such that for the control system \eqref{affine}:
    \begin{equation}
        \sup_{u} [L_{f}h(x) + L_{g}h(x)u] \geq -\alpha (h(x)),\label{cbf_cond}
    \end{equation}
    where $L_{f}h(x)$ and $L_{g}h(x)$ are the Lie-derivatives of $h(x)$ defined as:
    \begin{equation}
        L_f h(x) = \frac{\partial h}{\partial x}f(x), \qquad L_g h(x) = \frac{\partial h}{\partial x}g(x).
    \end{equation}

Now, consider another nonlinear control affine system:
\begin{align}
    \dot{x} = f(x) + g(x)u + p(x)d\label{dist_affine},
\end{align}
where $d \in \mathbb{R}^{m_2}$ is the disturbance with an assumption that it is upper bounded by a positive scalar $M$ (i.e., $\|d\| \leq M$).

\textbf{Definition 3}: (\textit{Modified from \cite{jankovic_robust_2018}}) A continuously differentiable function $h \; : \; \mathbb{R}^n \rightarrow \mathbb{R}$ is a \textit{Robust Control Barrier Function (RCBF)} if there exists an extended class $\mathcal{K}_{\infty}$ function $\alpha \; : \; \mathbb{R} \rightarrow \mathbb{R}$ such that for the control system \eqref{dist_affine}:
    \begin{align}
        \sup_{u} [L_{f}h(x) + L_{g}h(x)u - \|L_{p}h(x)\|M] \geq -\alpha (h(x)).\label{rcbf_cond}
    \end{align}

When taking multiple derivatives, as required by many high relative degree problems, the non-smoothness of \eqref{rcbf_cond} can be an issue. By noting the following upper bound of the non-smooth term with a smooth function:
\begin{align}
    \sqrt{\varepsilon + \|L_ph(x)\|^2} \geq \|L_ph(x)\|,
\end{align}
for $\varepsilon > 0$, we introduce the following extended definition of RCBFs.

\textbf{Definition 4}: A continuously differentiable function $h \; : \; \mathbb{R}^n \rightarrow \mathbb{R}$ is a \textit{smooth Robust Control Barrier Function (sRCBF)} if there exists an extended class $\mathcal{K}_{\infty}$ function $\alpha \; : \; \mathbb{R} \rightarrow \mathbb{R}$ such that for the control system \eqref{dist_affine}:
    \begin{align}
        &\sup_{u} \biggl[L_{f}h(x) + L_{g}h(x)u \nonumber\\
        &\qquad - M\sqrt{\varepsilon + \|L_ph(x)\|^2}\biggr]\geq -\alpha (h(x)).\label{srcbf_cond}
    \end{align}
for $\varepsilon > 0$.

With the necessary definitions, we now introduce our motivational problem.

\section{Motivational Problem}\label{motivation}
Consider a safety-critical problem in which an agent must maintain a safe distance from an obstacle. The agent will be modeled as a kinematic unicycle system:
\begin{align}
    \dot{x} &= u_{v} \cos\theta,\label{system_start}\\
    \dot{y} &= u_{v} \sin\theta,\\
    \dot{\theta} &= u_\theta,\label{system_start-1}
\end{align}
where $[x, y]^\top$ are the positional states, $\theta$ is the heading angle, $u_{v}$ is the forward drive input, and $u_{\theta}$ is the steering input. 

Suppose the agent lacks information about the obstacle's dynamics. As such, the agent will treat the obstacle dynamics as unkown disturbances, represented as follows:
\begin{align}
    \dot{x}_d &= d_x,\label{system_end-1}\\
    \dot{y}_d &= d_y\label{system_end},
\end{align}
where $[x_d, y_d]^\top$ denotes the obstacle's position and $d = [d_x, d_y]$ represents disturbances with an upper bound $\|d\| \leq M$, with $M \in \mathbb{R}^{+}$. 

To avoid collisions with the obstacle, a common choice for a candidate CBF is the Euclidean distance-based function:
\begin{align}
    h(x,y,x_d,y_d) = (x-x_d)^2 + (y-y_d)^2 - r^2.
\end{align}
Here, $r \in \mathbb{R}$ specifies the user-defined minimum safety distance.

This CBF, however, suffers from a mixed relative degree problem. Specifically, the first derivative of $h(x,y,x_d,y_d)$ does not provide direct access to $u_\theta$. As such, directly applying the standard CBF approach will result in an inefficient safety filter which cannot steer to achieve safety.

The approach proposed in this paper addresses this issue by using an extention of the CBF backstepping technique in \cite{krstic_nonovershooting_2006} to handle the mixed relative degree problem and the unknown disturbances effectively. We review this method in the next section.

\section{QP-Backstepping CBF Design}\label{qp-backstep}

As discussed, systems with relative degrees greater than one necessitate additional control design techniques. This paper builds on the concept of obtaining a target CBF of relative degree one through the CBF backstepping method as introduced in \cite{krstic_nonovershooting_2006}. In this section, we outline the methodology to obtain the target CBF, translating it into modern CBF terminology, and designing the controller through a Quadratic Program as first introduced in \cite{ames2014cbfqp_cruise}.

\noindent \textbf{CBF Backstepping Design Algorithm.} Consider again the control affine system \eqref{affine} with a desired candidate CBF $h_1(x)$ of relative degree $n > 1$ and with the following assumption:
\begin{assumption}\label{assume1}
    The function $h_1(x)$ is $n$-times differentiable and satisfies
    \begin{align}
        \frac{\partial h_1(x)}{\partial x} \neq 0, \quad \forall x \in \mathcal{C}.
    \end{align}
\end{assumption}

Taking the time derivative of $h_1(x)$ yields
\begin{align}
    \dot{h}_1(x) &= L_f h_1(x),\\
    &= -c_1 h_1(x) + \underbrace{c_1 h_1(x) + L_f h_1(x)}_{h_2(x)},
\end{align}
where $c_1 > 0$. It follows that if $h_2(x) \geq 0$, the inequality condition \eqref{cbf_cond} is satisfied for $h_1(x)$.

This leads to a new candidate CBF given by
\begin{align}
    h_2(x) = c_1 h_1(x) + L_f h_1(x).\label{h2}
\end{align}
However, we must ensure that the initial condition is within the safe set of $h_2(x)$ (i.e., $h_2(x_0) > 0$). To address this, we make the following assumption:

\begin{assumption}
    $h_1(x_0) > 0$.\label{assume_pos_init}
\end{assumption}

The assumption $h_1(x_0) \geq 0$ is common in standard CBF problems; however, it is important to note that we do not allow the initial condition to be on the boundary of the safe set $\mathcal{C}$ as designing a controller that maintains system safety when initialized on the boundary can be infeasible.

Given assumption \ref{assume_pos_init}, we choose
\begin{align}
    c_1 > \max\left\{0, \frac{-L_f h_1(x_0)}{h_1(x_0)}\right\},
\end{align}
which, from \eqref{h2}, ensures that $h_2(x_0) > 0$.

By iterating this method, we obtain the target CBF through a backstepping transformation defined as follows:
\begin{align}
    h_1(x) &= h_1(x),\\
    h_i(x) &= c_{i-1} h_{i - 1}(x) + L_f h_{i-1}(x),
\end{align}
for $i = \{2, \cdots, n\}$ where
\begin{align}
    c_{i-1} > \max\left\{0, \frac{-L_fh_{i-1}(x_0)}{h_{i-1}(x_0)}\right\}.
\end{align}

\noindent \textbf{Quadratic Program (QP).}
Having obtained the target CBF $h_n(x)$ with a relative degree of one, we design a safety override controller that adjusts the nominal control input $u_0(t)$ to the system when it is deemed unsafe.

This is accomplished by solving the  quadratic programming (QP) problem that minimizes the deviation of the overridden control from the nominal control while ensuring the following safety constraint:
\begin{align}
    L_f h_n(x) + L_g h_n(x) u \geq -c_n h_n(x),
\end{align}
where $c_n > 0$, is satisfied \cite{ames2014cbfqp_cruise}. That is, the optimization problem is formulated as:
\begin{align}
    &\qquad u = \min_{u} \| u - u_0 \|^2 \\
    \text{s.t.}\quad &L_f h_n(x) + L_g h_n(x) u \geq -c_n h_n(x).\label{qp_og}
\end{align}

\section{QP-Backstepping Robust CBF Design}\label{extention}

Now, we introduce an extension to the backstepping method to systems with additive disturbances. Consider the control affine system with disturbances as described in \eqref{dist_affine}, and let $h_1(x)$ be a desired candidate CBF with relative degree $n > 1$ and satisfies Assumption \ref{assume1}. 

By leveraging the sRCBF formulation, we systematically account for the worst-case effect of disturbances on the system safety, but avoid any singularities due to non-smoothness. 

For clarity, the following abbreviation will be used from this point:
\begin{align}
    \delta_k(x) \coloneqq \sqrt{\varepsilon_{k} + \|L_ph_{k}(x)\|^2},
\end{align}
with $k \in \mathbb{N}$.

We begin again by computing the time derivative of $h_1(x)$.

\begin{align}
    \dot{h}_1(x) &= L_f h_1(x) + L_p h_1(x) d, \\
    &\geq -c_1 h_1(x)\nonumber\\
    &\quad + \underbrace{c_1 h_1(x) + L_f h_1(x) - M\delta_1(x)}_{h_2(x)}.
\end{align}

Subsequently, by following the same reasoning and procedure outlined in Section \ref{qp-backstep}, we derive a similar chain of CBFs.
\begin{align}
    h_1(x) &= h_1(x), \label{chain_1}\\
    h_i(x) &= c_{i-1} h_{i - 1}(x) + L_f h_{i-1}(x) - M\delta_{i-1}(x),\label{chain_2}
\end{align}
with $\varepsilon_{i-1} > 0$ and
\begin{align}
    c_{i-1} > \max\left\{0, \frac{-L_f h_{i-1}(x_0) + M\delta_{i-1}(x_0)}{h_{i-1}(x_0)}\right\},\label{initial_gain}
\end{align}
for $i = \{2, \ldots, n\}$. With the target CBF established, we formulate the QP problem as follows:
\begin{align}\label{qp}
    &\qquad \qquad u = \min_{u} \| u - u_0 \|^2 \\
    \text{s.t.}\quad &L_f h_n(x) + L_g h_n(x)u - M\delta_n(x) \geq -c_n h_n(x).\label{qp1}
\end{align}
where $c_n > 0$.

By applying the Karush-Kuhn-Tucker (KKT) optimality conditions \cite{boyd2004convex}, this QP problem has an explicit solution in the form:
\begin{align}\label{control}
    u = \begin{dcases}
        u_0, &\eta(x,u_0) \geq 0,\\
        u_0 - (L_gh_n(x))^\top\frac{\eta(x,u_0)}{\|L_gh_n(x)\|^2}, &\text{otherwise},
    \end{dcases}
\end{align}
where
\begin{align}
    \eta(x,u_0) &\coloneqq L_f h_n(x) + L_g h_n(x)u_0\nonumber\\
    &\quad - M\delta_n(x) + c_n h_n(x).
\end{align}

Note, the effect of the worst-case disturbance is evident in the constraint \eqref{qp1}, where, compared to \eqref{qp_og}, the controller is required to override the nominal control much earlier than the boundary and becomes more conservative as the estimation of $M$ increases. Further, the smoothness factor $\varepsilon$ also contributes to the conservativeness and should be chosen such that the safety filter does not act too conservatively, but still handle the non-smoothness.

This conservativeness is inevitable when considering worst-case disturbances. However, approximating an upper bound on the disturbance is generally easier than estimating the obstacle dynamics, as shown in \cite{molnar_safety-critical_2023}. The maximum velocity of any obstacle can often be directly specified, either through physical limitations or environmental specifications. In contrast, estimating the obstacle dynamics may involve an indeterminate amount of error.

We are now ready to state the main result of this paper.

\begin{theorem}\label{thrm_1}
    Let the states of system \eqref{dist_affine} be initially within and not on the boundary of a desired safe set, i.e., $h_1(x_0) > 0$, and Assumptions \ref{assume1}--\ref{assume_pos_init} hold. Then, the control law \eqref{control}, from the chain of CBFs \eqref{chain_1}--\eqref{chain_2} with initial control gains \eqref{initial_gain} and $c_n > 0$, guarantees $h_1(x(t)) \geq 0$ for all $t \in [t_0, \infty)$, namely, the system remains safe for all nominal controls.
\end{theorem}

\begin{proof}
    The proof follows similar steps to the proof of Theorem 1 in \cite{abel2023prescribed_time}. 
    
    Given $ h_{i-1}(x_0) > 0 $, the initial gain \eqref{initial_gain} is designed to guarantee $ h_i(x_0) > 0 $. Since we know $ h_1(x_0) > 0 $, the gain choice
    \begin{align}
        c_1 > \max\left\{0, \frac{-L_f h_1(x_0) + M\delta_1(x_0)}{h_1(x_0)}\right\},
    \end{align}
    ensures $ h_2(x_0) > 0 $. Thus, by induction, we can conclude $ h_i(x_0) > 0 $ for $ i = \{2, \cdots, n\}$.

    Then, taking the time derivative of the chain of CBFs \eqref{chain_1}--\eqref{chain_2} yields
    \begin{align}
        \frac{d}{dt} h_j(x(t)) &= -c_j h_j(x(t)) + h_{j+1}(x(t)),\\
        \frac{d}{dt} h_n(x(t)) &=L_f h_n(x(t))\nonumber\\
        &\quad + L_g h_n(x(t))u - M\delta_n(x)\\
        &\geq -c_n h_n(x(t)),
    \end{align}
    for $j = \{1, \cdots, n-1\}$. Applying the Comparison lemma in parallel with the variation of constants formula, for $ t \in [t_0, \infty) $, results in
    \begin{align}
        h_j(x(t)) &= h_j(x(t_0)) e^{-c_j (t - t_0)}\nonumber\\
        &\quad + \int_{t_0}^{t} h_{j+1}(x(\tau)) e^{-c_j (t - \tau)} d\tau, \label{h_j}\\
        h_n(x(t)) &\geq h_n(x(t_0)) e^{-c_n (t - t_0)}, \label{h_n}
    \end{align}
    for $j = \{1, \cdots, n-1\}$. 
    
    As established before, $ h_n(x(t_0)) > 0 $, which implies from \eqref{h_n} that $ h_n(x(t)) > 0 $ for all $ t \in [t_0, \infty) $. Now, substituting \eqref{h_n} into \eqref{h_j} for $ j = n - 1 $ results in
    \begin{align}
        h_{n-1}(x(t)) &\geq \underbrace{h_{n-1}(x(t_0)) e^{-c_{n-1} (t - t_0)}}_{>0}\nonumber\\
        &\quad + \underbrace{h_n(x(t_0)) \int_{t_0}^{t} e^{-(c_n + c_{n-1})(t - \tau)(\tau - t_0)} d\tau}_{\geq 0}\\
        &\geq h_{n-1}(x(t_0)) e^{-c_{n-1} (t - t_0)} > 0.\label{h_n_1}
    \end{align}
    
    By using backward strong induction with \eqref{h_n} and \eqref{h_n_1}, one can observe that
    \begin{align}
        h_1(x(t)) \geq h_1(x(t_0)) e^{-c_1 (t - t_0)}.
    \end{align}
    Thus, $ h_1(x(t)) \geq 0 $ for all $ t \in [t_0, \infty) $.
\end{proof}

With this new approach established, we return to our motivational problem.

\section{Collision Avoidance with Unknown Obstacle Dynamics}\label{application}

\subsection{Augmented System}
Recall our original problem of a unicycle agent with the goal of avoiding an unknown moving obstacle. First, we build an augmented system which includes both the agent's states whose dynamics are governed by \eqref{system_start}--\eqref{system_start-1} and the obstacle states driven by an unknown disturbance  defined in  \eqref{system_end-1}--\eqref{system_end}. Furthermore, to avoid mixed relative degrees of the system  \eqref{system_start}--\eqref{system_start-1} with respect to the positional constraint, we introduce an additional state $v$ to raise the relative degree of the input signal $u_{v}$, essentially modifying the model from a unicycle to a simplified bicycle model \cite{rahman_driver_2021,rajamani_vehicle_2012}. This results in a homogeneous relative degree two problem. 

The full system is as follows:
\begin{align}
    \dot{\mathbf{x}}
    =
    \begin{bmatrix}
        \dot{x}\\
        \dot{y}\\
        \dot{v}\\
        \dot{\theta}\\
        \dot{x}_d\\
        \dot{y}_d
    \end{bmatrix}
    =
    \underbrace{
    \begin{bmatrix}
        v\cos\theta\\
        v\sin\theta\\
        0\\
        0\\
        0\\
        0
    \end{bmatrix}}_{f(\mathbf{x})}
    +
    \underbrace{
    \begin{bmatrix}
        0\\
        0\\
        u_{v}\\
        u_{\theta}\\
        0\\
        0
    \end{bmatrix}}_{g(\mathbf{x})u}
    +
    \underbrace{
    \begin{bmatrix}
        0\\
        0\\
        0\\
        0\\
        d_x\\
        d_y
    \end{bmatrix}}_{p(\mathbf{x})d}
\end{align}

\begin{figure}[ht!]
    \centering
    \includegraphics[width=0.95\linewidth]{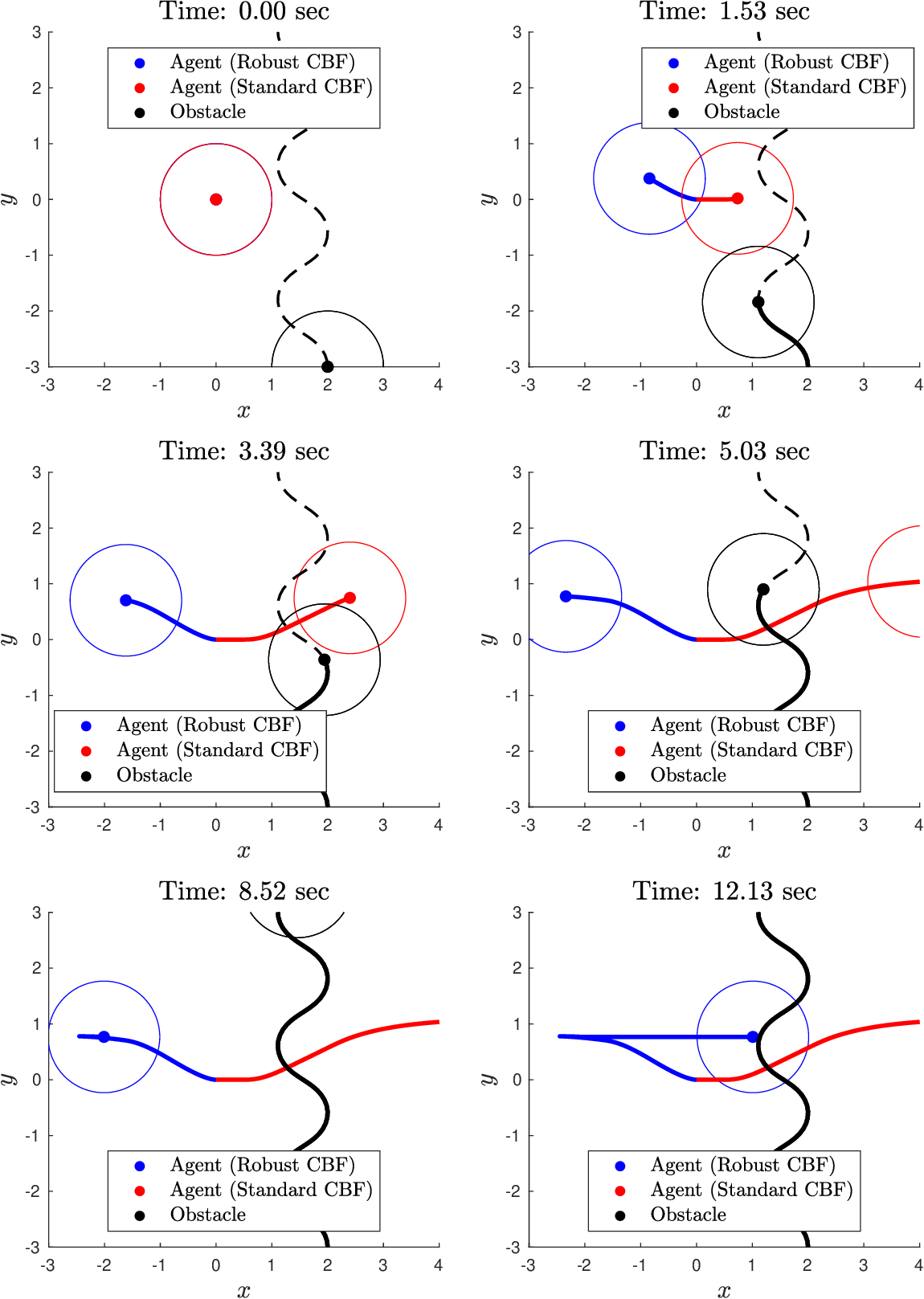}
    \caption{System trajectory of agents in the presence of an obstacle (black) with unknown dynamics. The agent using the standard CBF (red) fails to avoid collision, while the agent considering the worst-case disturbance (blue) remains safe from the obstacle.}
    \label{fig:sim}
\end{figure}

\subsection{Safety Filter Design}
Then, consider again the desired candidate CBF:
\begin{align}
    h_1(\mathbf{x}) &= (x - x_d)^2 + (y - y_d)^2 - r^2
\end{align}
and we proceed with the CBF backstepping design method. Taking the time derivative of $h_1(\mathbf{x})$,
\begin{align}
    \dot{h}_1(\mathbf{x}) &= 2(x-x_d)v\cos\theta + 2(y-y_d)v\sin\theta\nonumber\\
    &\quad -2(x-x_d)d_x - 2(y-y_d)d_y,\\
    &\geq 2(x-x_d)v\cos\theta \nonumber\\
    &\quad  + 2(y-y_d)v\sin\theta - M\delta_1(\mathbf{x}),\\
    &= -c_1 h_1(\mathbf{x})\nonumber\\
    &\quad + \underbrace{c_1 h_1(\mathbf{x})  + L_fh_1(\mathbf{x}) -M\delta_1(\mathbf{x})}_{h_2(\mathbf{x})},
\end{align}
yields
\begin{align}
    h_2(\mathbf{x}) &\coloneqq c_1 h_1(\mathbf{x})  + L_fh_1(\mathbf{x}) -M\delta_1(\mathbf{x}),
\end{align}
where 
\begin{align}
    &\|L_ph_{1}(\mathbf{x})\| = 4(x-x_d)^2 + 4(y-y_d)^2\label{Lph}.
\end{align}

Now, we choose 
\begin{align}
    c_1 > \max\left\{0, \frac{-L_fh_1(\mathbf{x}_0) + M\delta_1(\mathbf{x}_0)}{h_1(\mathbf{x}_0)}\right\},
\end{align}
such that $h_2(\mathbf{x}_0) > 0$. We have shown in Theorem \ref{thrm_1} that enforcing $h_2(\mathbf{x}) \geq 0$ will ensure that the system will be safe. Taking the time derivative of $h_2(\mathbf{x})$ yields
\begin{align}
    \dot{h}_2(\mathbf{x}) &= L_f h_2(\mathbf{x}) + L_g h_2(\mathbf{x}) u + L_p h_2(\mathbf{x}) d, \\
    &\geq L_f h_2(\mathbf{x}) + L_g h_2(\mathbf{x}) u  - M\delta_2(\mathbf{x}).
\end{align}

From  \eqref{qp} and \eqref{qp1}, our  QP problem is formulated as follows:
\begin{align}
    &\qquad \qquad u = \min_{u} \| u - u_0 \|^2 \\
    \text{s.t.}\quad &L_f h_2(\mathbf{x}) + L_g h_2(\mathbf{x}) u - M\delta_2(\mathbf{x}) \geq -c_2 h_2(\mathbf{x})\label{qp_uni},
\end{align}
with the explicit solution in the form \eqref{control}. The detailed expression for each of the terms is given in the Appendix.

\begin{figure}[t]
    \centering
    \includegraphics[width=0.9\linewidth]{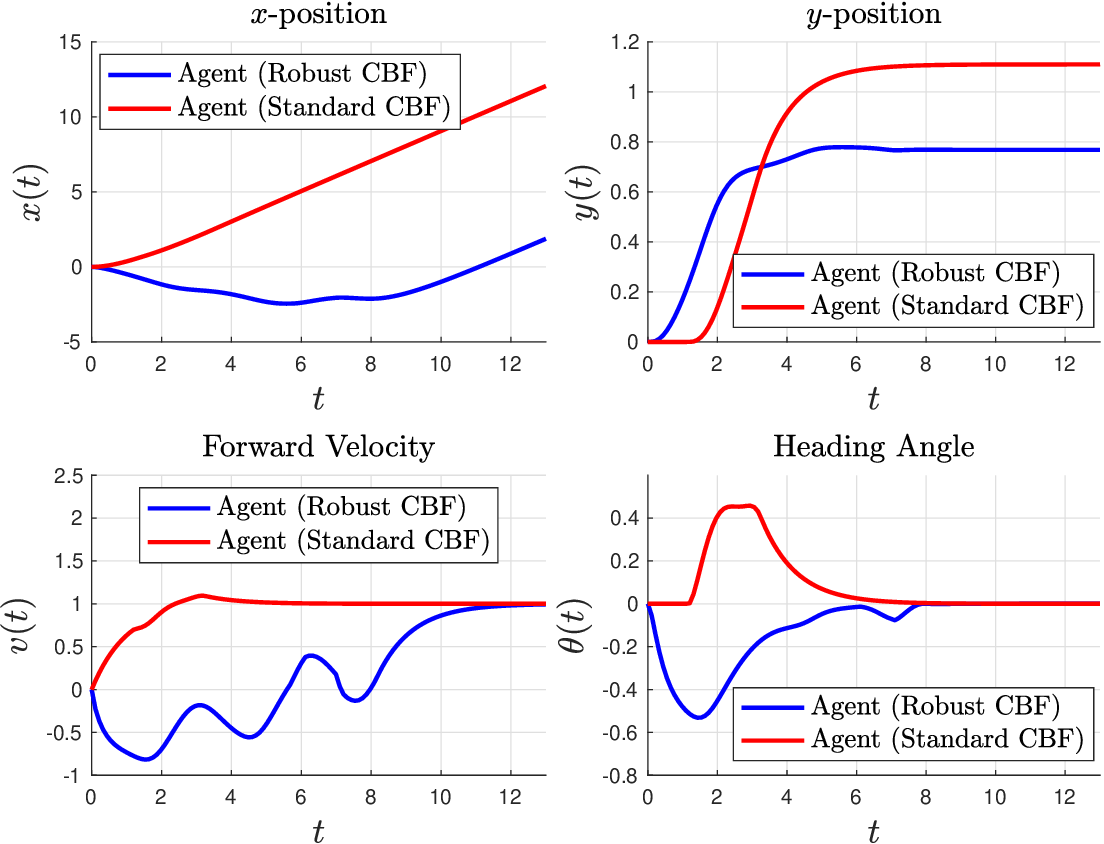}
    \caption{System states of the agents representing the positional states $[x(t),y(t)]$, forward velocity $v(t)$, and the heading angle $\theta(t)$.}
    \label{fig:states}
\end{figure}

\subsection{Simulations: unknown moving obstacle avoidance}
We present a simulation to demonstrate the effectiveness of the proposed method. The simulation involves two agents: one agent that uses the standard CBF backstepping method from \cite{krstic_nonovershooting_2006}, while the other employs the new sRCBF-based CBF backstepping method, which accounts for obstacle dynamics as unknown disturbances. The agents were both initialized at $[x(0), y(0), v(0), \theta(0)] = [0,0,0,0]$ and the safety distance was chosen to be $r = 2$ with a smoothness factor $\varepsilon_1 = \varepsilon_2 = 0.01$.

The obstacle was modeled as another unicycle model with the dynamics
\begin{align}
    \dot{x}_d &= v_d(t)\cos\theta_d,\\
    \dot{y}_d &= v_d(t)\sin\theta_d,\\
    \dot{\theta}_d &= \omega_d(t),
\end{align}
and initialized at $[x_{d}(0),y_{d}(0),\theta_{d}(0)] = [2,-3,\frac{\pi}{2}]$ with $v_d(t) = 1$ and $\omega_d(t) = 2\cos(2t)$. Consequently, the provided upper bound to the agent was $M = 1$ and the initial control gain for both agents were chosen to be $c_1 = 3$ and $c_2 = 1$.

The nominal control objective for the agents was defined to move with a heading angle of $\theta = 0$ and a constant velocity of $v = 1$, i.e., to move to the right. Hence, the nominal control inputs were defined as $u_{v0} = -k_1(v - 1)$ and $u_{\theta0} = -k_2 \theta$, where $k_1 = k_2 = 1$.

Fig. \ref{fig:sim} presents the resulting system trajectories and Fig. \ref{fig:states} shows the evolution of the agents' states under the standard CBF backstepping design and the sRCBF backstepping design. In Fig. \ref{fig:sim}, the agent which does not consider the dynamics of the obstacle violates safety, as seen at $t = 3.39s$, resulting in a collision. In contrast, the agent accounting for the worst-case disturbance first avoided the obstacle around $t = 1.53s$, waited for it to clear approximately at $t = 5.03s$, and then resumed its nominal task once it was safe after $t = 8.52s$. The sRCBF agent initially retreats backwards and only proceeds when the obstacle was well clear, which validates our approach.

\section{Conclusion}
This paper presented a novel method of addressing high relative degree safety constraints with disturbances by integrating the sRCBF definition to the CBF backstepping technique. By incorporating the worst-case disturbance effects directly into the control design, our method simplified the requirements for uncertain environments. In the context of avoiding an obstacle with unknown dynamics, our method provides a safety guarantee with a relatively easy-to-obtain upper bound of the obstacle velocity. The proposed method was validated through a simulation, which demonstrated a successful collision avoidance with an uncertain obstacle. Future work will explore further applications of this method as well as further improvements in collision avoidance when given more detailed information about the obstacle dynamics, such as the angular velocity.

\section*{Appendix}
Abbreviated expression in the QP constraint \eqref{qp_uni}.

\begin{align}
    L_fh_2(\mathbf{x}) &= 2c_1\left((x-x_d)v\cos\theta + (y-y_d)v\sin\theta\right) + 2v^2\nonumber\\
    &\quad - \frac{4M}{\delta_1(\mathbf{x})}\bigl((x-x_d)v\cos\theta + (y-y_d)v\sin\theta\bigr)\\
    L_gh_2(\mathbf{x})u &= \Bigl(2(x-x_d)\cos\theta + 2(y-y_d)\sin\theta\Bigr)u_{v}\nonumber\\
    &\quad+\Bigl(2(y-y_d)v\cos\theta - 2(x-x_d)v\sin\theta\Bigr)u_{\theta}\\
    L_ph_2(\mathbf{x})d &= \Biggl(\left(\frac{4M}{\delta_1(\mathbf{x})}-2c_1\right)(x-x_d) -2v\cos\theta \Biggr)d_x\nonumber\\
    &\quad +\Biggl(\left(\frac{4M}{\delta_1(\mathbf{x})}-2c_1\right)(y-y_d) -2v\sin\theta \Biggr)d_y
\end{align}


\balance
\bibliographystyle{root}
\bibliography{root}
\end{document}